\documentclass[12pt]{amsart}
\usepackage{amssymb}
\usepackage{hyperref}

\newtheorem{thm}{Theorem}[section]
\newtheorem{lem}[thm]{Lemma}
\newtheorem{cor}[thm]{Corollary}
\newtheorem{prop}[thm]{Proposition}
\newtheorem{ex}[thm]{Example}

\newtheorem*{prob*}{Open problem}

\theoremstyle{definition}

\newtheorem{defi}[thm]{Definition}

\theoremstyle{remark}

\newtheorem{rem}[thm]{Remark}
\newtheorem*{rem*}{Remark}

\newcommand{\kringel}{\mathbin{\raise1pt\hbox{$\scriptstyle\circ$}}}
\newcommand{\pkt}{\mathbin{\raise0pt\hbox{$\scriptstyle\bullet$}}}

\newcommand{\C}{\mathbb{C}}

\newcommand{\N}{\mathbb{N}}

\newcommand{\R}{\mathbb{R}}

\newcommand{\ad}{{\rm ad}}

\newcommand{\End}{{\rm End}}
\newcommand{\Der}{{\rm Der}}

\newcommand{\La}{\mathfrak{a}}

\newcommand{\Lf}{\mathfrak{f}}
\newcommand{\Lg}{\mathfrak{g}}

\newcommand{\Ll}{\mathfrak{l}}

\newcommand{\Ls}{\mathfrak{s}}

\newcommand{\LD}{\mathfrak{D}}

\newcommand{\CB}{\mathcal{B}}
\newcommand{\CC}{\mathcal{C}}

\newcommand{\al}{\alpha}

\newcommand{\Ga}{\Gamma}

\renewcommand{\phi}{\varphi}

\newcommand{\Inn}{{\rm Inn}}
\newcommand{\AID}{{\rm AID}}
\newcommand{\CAID}{{\rm CAID}}

\begin{document}


\title[Almost Inner Derivations]{Almost inner derivations of Lie algebras II}

\author[D. Burde]{Dietrich Burde}
\author[K. Dekimpe]{Karel Dekimpe}
\author[B. Verbeke]{Bert Verbeke}
\address{Fakult\"at f\"ur Mathematik\\
Universit\"at Wien\\
  Oskar-Morgenstern-Platz 1\\
  1090 Wien \\
  Austria}
\email{dietrich.burde@univie.ac.at}
\address{Katholieke Universiteit Leuven Kulak\\
E. Sabbelaan 53 bus 7657\\
8500 Kortrijk\\
Belgium}
\email{karel.dekimpe@kuleuven.be}
\email{bert.verbeke@kuleuven.be}

\date{\today}


\begin{abstract}
We continue the algebraic study of almost inner derivations of Lie algebras over a field of characteristic zero 
and determine these derivations for free nilpotent Lie algebras, for almost abelian Lie algebras, 
for Lie algebras whose solvable radical is abelian and for several classes of filiform nilpotent Lie algebras.
We find a family of $n$-dimensional characteristically nilpotent filiform Lie algebras $\Lf_n$, 
for all $n\ge 13$, all of whose derivations are almost inner. Finally we compare the almost inner derivations of 
Lie algebras considered over two different fields $K\supseteq k$ for a finite-dimensional field extension. 
\end{abstract}

\maketitle
\section{Introduction}

Almost inner automorphisms of Lie groups and almost inner derivations of Lie algebras have been introduced 
by Gordon and Wilson \cite{GOW} in the study of isospectral deformations of compact solvmanifolds. 
They constructed isospectral but non-isometric compact Riemannian manifolds of the form $G/\Ga$, with a simply connected 
exponential solvable Lie group $G$, and a discrete cocompact subgroup $\Ga$ of $G$. This construction relies on almost 
inner automorphisms and almost inner derivations. Larsen \cite{LAR} studied algebraic groups for which every
two almost conjugated homomorphisms are globally conjugated. This is closely related to the question whether
or not a compact group can be the common covering space of a pair of non-isometric isospectral manifolds. \\
The concept of ``almost inner'' automorphisms and derivations, almost homomorphisms, or almost conjugate subgroups 
arises in many areas of algebra and geometry. However, a systematic algebraic study was not done so far. 
So we started an investigation of almost inner derivations of Lie algebras in \cite{BU55}. The aim of this paper is to
continue this study and prove further results on almost inner derivations for certain classes of Lie algebras. \\[0.2cm]
We recall the following definitions. Let $\Lg$ be a finite-dimensional Lie algebra over a field $K$ and $\Der(\Lg)$ its 
derivation Lie algebra. A derivation $D\in \Der(\Lg)$ is said to be {\em almost inner}, if $D(x)\in [\Lg,x]$ for all $x\in \Lg$. 
The space of all almost inner derivations of $\Lg$ is denoted by $\AID(\Lg)$. 
The subspace $\AID(\Lg)$ becomes a Lie subalgebra of $\Der(\Lg)$ by the Lie bracket $[D,D']=DD'-D'D$.
We denote the Lie subalgebra of inner derivations by $\Inn(\Lg)$. An almost inner derivation $D\in \AID(\Lg)$ is 
called {\em central almost inner} if there exists an $x\in \Lg$ such that $D-\ad (x)$ maps $\Lg$ to the center $Z(\Lg)$. 
We denote the subalgebra of central almost inner derivations of $\Lg$ by $\CAID(\Lg)$. \\[0.2cm]
The paper is structured as follows. In the second section we show that every almost inner derivation of a free nilpotent
Lie algebra over a field of characteristic zero is inner. Similarly, in the third section, we show that every
almost inner derivation of an almost abelian Lie algebra is inner. In the fourth section we compute the almost inner
derivations for certain classes of filiform nilpotent Lie algebras, e.g., for the Witt algebras $W_n$ and for a family
$\Lf_n$, $n\ge 13$, which is closely related to $W_n$. We show that the Witt algebra has $4$ linearly independent outer
derivations, where $3$ of them are almost inner. For $\Lf_n$ we show that every derivation is almost inner.
This comes as a surprise and it is the first known family of nilpotent Lie algebras, where all derivations are
almost inner. In the fifth section we show that every almost inner derivation is inner for Lie algebras whose
solvable radical is abelian over an algebraically closed field of characteristic zero. Here we use results 
about ``distinguished'' elements
in semisimple Lie algebras, whose centralizers consist entirely of nilpotent elements. \\
The last two sections contain results about almost inner derivations of Lie algebras considered over two different fields 
$K\supseteq k$. In particular we show
that if a Lie algebra viewed over $K$ admits a non-trivial almost inner derivation, then so does also the Lie
algebra viewed over the smaller field $k$. However, the converse does not hold in general. We can construct new almost 
inner derivations using field extensions.

\section{Free nilpotent Lie Algebras}

In this section we will prove that a free nilpotent Lie algebra over a field of characteristic zero does not admit any 
non-trivial almost inner derivations. We will use the following lemma, which can be shown by induction on $k$.

\begin{lem}
\label{lem: induction-lemma}
Let $k$ be a non-negative integer and $V$ be a vector space over a field $K$ of characteristic zero. Consider a 
sequence $v_0, v_1, v_2, \ldots$ in $V$. Suppose that there exist $a_0, a_1, \ldots, a_k \in V$ such that 
\begin{equation*}
v_{n + 1} - v_n = \sum_{j = 0}^k n^j a_j 
\end{equation*} 
for all  $n \in \N$. Then there exist vectors $b_0, b_1, \ldots, b_{k + 1} \in V$ such that
\begin{align*}
v_n & = \sum_{j = 0}^{k + 1} n^j b_j \quad \forall\, n \in \N, \\
a_k & \neq 0 \Longrightarrow b_{k + 1} \neq 0.
\end{align*}
\end{lem}

Consider the free Lie algebra $\Lg$ on two generators $x_1$ and $x_2$. Define $\Lg_1$ as the vector space spanned by the
generators  $x_1$ and $x_2$ and $\Lg_n$, for $n \geq 2$, as the subspace of $\Lg$ generated by all Lie brackets of 
length $n$ in the generators $x_1$ and $x_2$. Denote further $\Lg_{i,j}$ for the subspace of $\Lg$ generated by all 
Lie brackets in the generators where the first generator $x_1$ appears $i$ times and the second one $x_2$ appears $j$ times. 
It is clear that
\begin{equation*}
\Lg = \bigoplus_{n=1}^\infty \Lg_n \quad \text{ and } \quad \Lg_n = \bigoplus_{i = 1}^{n-1} \Lg_{i,n-i} \; (\text{for } n \geq 2).
\end{equation*} 
Moreover, it holds that 
\begin{equation*}
[ \Lg_i, \Lg_j]\subseteq \Lg_{i+j} \quad \text{ and } \quad [\Lg_{i,j}, \Lg_{p,q}]\subseteq \Lg_{i+p,j+q}. 
\end{equation*} 

We are interested in the equation
\begin{equation*}
[x_1,x] + [x_2,y] = 0
\end{equation*} in the variables $x$ and $y$, which was studied in \cite{RES}.
Let 
\begin{equation*}
V = \{ (x,y) \in \Lg \times \Lg \mid [x_1,x] + [x_2,y] = 0\}
\end{equation*} be the solution space of the equation. Note that $V$ is a vector space. For each $n \in \N_0$, 
we define $V_n = V \cap (\Lg_n \times \Lg_n)$. Consider now the maps
\begin{align*}
\varphi_n & : V_n \to \Lg_n, (x,y) \mapsto x, \\
\psi_n & : V_n \to \Lg_n, (x,y) \mapsto y.
\end{align*} 
Denote $\varphi_n(V_n) = V_n^x$ and $\psi_n(V_n) = V_n^y$. We assert that $\tilde{\varphi}_n: V_n \to V_n^x: (x,y) 
\mapsto \varphi_n(x,y)$ is an isomorphism for all $n \geq 2$. It is obvious that the map $\tilde{\varphi}_n$ is linear. 
Surjectivity follows by construction. Suppose that $\tilde{\varphi}_n$ is not injective, then there is a solution 
$(0,0) \neq (0,y) \in V_n$, which means that $[x_2,y] = 0$. Therefore, $y = 0$ and we have a contradiction. 
Analogously, also $\tilde{\psi}_n: V_n \to V_n^y: (x,y) \mapsto \psi_n(x,y)$ is an isomorphism. Hence, for all $n\geq 2$, 
there is a vector space isomorphism $\sigma: V_n^x \to V_n^y$ such that $[x_1,x] + [x_2,\sigma(x)] = 0$ for all $x \in V_n^x$. 
Note that under this isomorphism $\sigma(V_n^x\cap \Lg_{i,n-i}) =V_n^y\cap \Lg_{i+1,n-i-1}$ for all $0 < i <n$.  
Denote further by $\Lg^1=\Lg$ and $\Lg^i=[\Lg,\Lg^{i-1}]$, $i\ge 2$ the terms of the lower central series.

\begin{thm}
Let $K$ be a field of characteristic zero and let $\Lf_{r,c}$ be the free $c$--step nilpotent Lie algebra over $K$ 
on $r$ generators. Then we have $\AID(\Lf_{r,c}) = \Inn(\Lf_{r,c})$.
\end{thm}

\begin{proof}
We prove this theorem by induction on the nilpotency class $c$. The case $c=1$ is clear and the cases $c=2$ and $c=3$ 
were already treated in \cite{BU55}. So let $c \geq 3$ and assume that the theorem holds for $\Lf_{r,c}$.
Consider now $\Lf_{r,c+1}$ with generators $x_1, x_2, \ldots, x_r$. Let $D \in \AID(\Lf_{r,c+1})$ be an almost inner 
derivation of $\Lf_{r,c+1}$. We will prove that $D$ is in fact inner. It is clear that $D$ induces an almost inner 
derivation on 
\[
\Lf_{r,c+1}/\Lf_{r,c+1}^{c+1} \cong \Lf_{r,c}. 
\]
By the induction hypothesis, this is an inner derivation. Hence, by changing $D$ up to an inner derivation, we may 
assume that $D(\Lf_{r,c+1}) \subseteq \Lf_{r,c+1}^{c+1} = Z(\Lf_{r,c+1})$, which means that $D \in \CAID(\Lf_{r,c+1})$. 
Hence, there exists $v \in \Lf_{r,c+1}^c$ such that $D(x_1) = [x_1,v]$. By replacing $D$ by $D + \ad(v)$, we can assume 
that $D$ is an almost inner derivation of $\Lf_{r,c+1}$ with $D(x_1) = 0$. Further, for all $x \in \Lf_{r,c+1}$, we have 
$D(x) = [x, w(x)]$, with $w(x) \in \Lf_{r,c+1}^c$. It suffices to prove that $D(x_i) = 0$ for all $i \in \{2,\ldots,r\}$.
We first look at $x_2$. For each $n \in \N$, there exists a $w_n \in \Lf_{r,c+1}^c$ such that 
\begin{equation}
\label{eq: D(nx_1 + x_2) = [n x_1 + x_2, w_n]}
D(nx_1 + x_2) = [n x_1 + x_2, w_n],
\end{equation} because $D$ is almost inner. We can assume without loss of generality that $w_n$ is a linear combination 
of Lie brackets of length $c$ in the generators (and does not contain a component using Lie brackets of length  $c + 1$).
By linearity, we also have that 
\begin{equation*}
D(nx_1 + x_2) = n D(x_1) + D(x_2) = D(x_2).
\end{equation*}
The two observations above imply that the equation
\begin{equation}
\label{eq: equation with w_0}
[n x_1 + x_2, w_n] = [m x_1 + x_2, w_m] = [x_2, w_0]
\end{equation} holds for all $n,m \in \N$. We consider $[n x_1 + x_2, w_n] + [x_2, -w_0] = 0$ as an equation in the free 
Lie algebra $\Lf_r$ on $r$ generators. For $n \neq 0$, define $x_1' := n x_1 + x_2$. It is clear that $x_1',x_2,\ldots, x_r$ 
is also a free generating set for the free Lie algebra $\Lf_r$. \\[0.2cm]
It follows from \cite[section 5]{RES} that $w_n, w_0 \in \langle x_1', x_2 \rangle = \langle x_1, x_2 \rangle$ 
for all $n \in \N_0$, where $\langle x_1,x_2\rangle$ denotes the Lie algebra generated by $x_1$ and $x_2$. 
This means that $w_n$ can be written as $w_n = \sum_{i = 1}^{c - 1} v_i(n)$, where $v_i(n)$ is a linear combination of Lie 
brackets where $x_2$ and $x_1$ appear $i$ respectively $c - i$ times. We can assume without loss of generality that 
we work in $\Lg$, the free Lie algebra on two generators $x_1$ and $x_2$ (and so $v_i(n)\in \Lg_{c-i,i}$, using the 
notations introduced above this theorem).
To prove that $D(x_2) = 0$, it suffices by equation (\ref{eq: D(nx_1 + x_2) = [n x_1 + x_2, w_n]}) to show that 
$w_0 = 0$. Suppose on the contrary that $w_0 \neq 0$. Define 
\begin{equation*}
k = \max\{i \in \N \mid \exists n \in \N \text{ with } v_i(n) \neq 0\},
\end{equation*} then $w_n = \sum_{i = 1}^k v_i(n)$ and there exists an $n \in \N$ such that $v_k(n) \neq 0$. It follows 
from equation (\ref{eq: equation with w_0}) that $[(n+1)x_1 + x_2, w_{n+1}] = [nx_1 + x_2, w_n]$, which implies that
\begin{equation}
\label{eq: main equation}
[x_1,(n+1)w_{n+1} - n w_n] + [x_2, w_{n+1} - w_n] = 0,
\end{equation} with $n \in \N$.
Note that this consists in fact of several equations (one per bi-degree $(i,j)$ with $i+j=c$).

We will now prove by induction on $p$ that for all $p \in \{0,\ldots,k-1\}$ and all $0\leq i \leq p$, there exist $b_{p,i} \in \Lg_{c-k+p,k-p}$, 
with $b_{p,p} \neq 0$ such that
\begin{equation*}
v_{k - p}(n) = n^p b_{p,p} + n^{p - 1} b_{p,p-1} + \ldots + n b_{p,1} + b_{p,0}.
\end{equation*}

{\it Basis step $p = 0$}: we first consider the component of equation (\ref{eq: main equation}) with in total $k + 1$ 
appearances of $x_2$, i.e.\, the bi-degree $(c-k,k)$--part. This gives
\begin{equation*}
[x_1,0] + [x_2,v_k(n+1) - v_k(n)] = 0.
\end{equation*} Hence, $v_k(n + 1) - v_k(n) = 0$, which means that $v_k(n)$ is a constant $b_{0,0} \neq 0$ and belongs 
to $\Lg_{c-k,k}$. Therefore, we have that $w_n = \left(\sum_{i = 1}^{k - 1} v_i(n) \right) + b_{0,0}$. \\[0.2cm]
{\it Induction step}: we assume that the assertion holds for a given  $p < k - 1$. Hence, there exist 
$b_{p,p}, b_{p,p-1},\ldots, b_{p,0}\in \Lg_{c-k+p,k-p}$ with $b_{p,p} \neq 0$ such that
\begin{equation*}
v_{k-p}(n) = n^p b_{p,p} + n^{p - 1} b_{p,p-1} + \ldots + b_{p,0}.
\end{equation*}
From the component of equation (\ref{eq: main equation}) with $k-p$ appearances of $x_2$, it follows that 
\begin{equation*}
[x_1, (n+1) v_{k-p}(n+1) - n v_{k-p}(n)] + [x_2, v_{k-p-1}(n+1) - v_{k-p-1}(n)] = 0.
\end{equation*} Hence, 
\begin{eqnarray*}
\lefteqn{v_{k-p-1}(n+1) - v_{k-p-1}(n)}\\ 
& =& \sigma\big((n+1) v_{k-p}(n+1) - n v_{k-p}(n) \big)\\
& = &\sigma\big((n+1)^{p+1}b_{p,p} + (n+1)^p b_{p,p-1} + \ldots + (n+1) b_{p,0} \\
& & \quad - n^{p+1} b_{p,p} - n^p b_{p,p-1} - \ldots - n b_{p,0}\big)\\
& = &(n+1)^{p+1} \sigma(b_{p,p}) + (n+1)^p \sigma(b_{p,p-1}) + \ldots + (n+1) \sigma(b_{p,0}) \\
& & \quad - n^{p+1} \sigma(b_{p,p}) - n^p \sigma(b_{p,p-1}) - \ldots - n \sigma(b_{p,0}),
\end{eqnarray*} which can be written as the sum of $n^p (p+1) \sigma(b_{p,p})$ and terms of lower degree. 
Since $b_{p,p} \neq 0$, also $(p+1) \sigma(b_{p,p})\neq 0$ holds. Note that all $\sigma(b_{p,i})$ belong to 
$\Lg_{c-k+p+1,k-p-1}$. Hence, it follows from Lemma \ref{lem: induction-lemma} that there exist 
$b_{p+1,p+1},\ldots, b_{p+1,0}\in\Lg_{c-k+p+1,k-p-1} $ with $b_{p+1,p+1} \neq 0$ such that
\begin{equation*}
v_{k-p-1}(n) = n^{p+1} b_{p+1,p+1} + n^p b_{p+1,p} + \ldots + b_{p+1,0},
\end{equation*}
which concludes the proof of our claim on the form of the $v_{k-p}(n)$. \\[0.2cm]
The above assertion implies that for all $n \in \N$, the equation $v_1(n) = n^{k-1} b_{k-1,k-1} + \ldots + b_{k-1,0}$ 
holds, where $b_{k-1,i} \in \Lg_{c-1,1}$ and $b_{k-1,k-1} \neq 0$. \\[0.2cm]
We now look at the term of equation (\ref{eq: main equation}) with exactly one factor of $x_2$. We then have
\begin{equation*}
[x_1,(n+1)v_1(n+1) - n v_1(n)] + [x_2,0] = 0
\end{equation*} and thus
\begin{equation*}
(n+1) v_1(n+1) - n v_1(n) = 0.
\end{equation*} This implies that
\begin{align*}
0 & = (n+1) \left((n+1)^{k-1} b_{k-1,k-1} + \ldots + b_{k-1,0} \right) - n \left(n^{k-1} b_{k-1,k-1} + \ldots + b_{k-1,0} \right)\\
& = (n+1)^k b_{k-1,k-1} + \ldots + (n+1) b_{k-1,0} - n^k b_{k-1,k-1} - \ldots - n b_{k-1,0}\\
& = k n^{k-1} b_{k-1,k - 1} + \sum_{i=2}^k {k \choose i} n^{k-i} b_{k-1,k-1} + \sum_{i=1}^{k-1} {k-1 \choose i} n^{k-1-i} b_{k-1,k-2} 
+ \ldots + b_{k-1,0}.
\end{align*} Hence, we can write $0$ as a sum of $k n^{k-1} b_{k-1,k - 1}$ and some terms of lower degree. This equation has 
to hold for all $n \in \N$, which implies that $kb_{k-1,k-1} = 0$. Since we work in a field of characteristic zero, this 
gives a contradiction, because $b_{k-1,k-1} \neq 0$. Hence, $w_0 = 0$. It now follows from equation 
(\ref{eq: D(nx_1 + x_2) = [n x_1 + x_2, w_n]}) that $D(x_2) = 0$. By a similar reasoning, we find that 
$D(x_i)=0$ for all $i \in \{3,\ldots,r\}$. This finishes the proof.
\end{proof}

\section{Almost abelian Lie algebras}

The aim of this section is to show the following result.

\begin{thm}
Let $\Lg$ be a finite-dimensional Lie algebra over a field $K$ containing an abelian ideal of codimension one.
Then $\AID(\Lg)=\Inn(\Lg)$.
\end{thm}

\begin{proof}
As $\Lg$ has a codimension one abelian ideal, it holds that $\Lg\cong K^n \rtimes_\varphi K$ for some Lie algebra morphism
$\varphi:K \to \Lg\Ll_n(K)$. We use $t$ to denote a basis vector of $K$. With respect to a suitable basis of $K^n$, we may 
assume that $\varphi(t)$ is in rational canonical form. This means that there is a basis 
$e_{i,j}$ ($1\leq i \leq r$, $1\leq j \leq k_i$) of $K^n$ such that 
\[ \varphi(t) =\begin{pmatrix}
C_1 & 0 & \cdots & 0\\
0 & C_2 & \cdots & 0\\
\vdots & \vdots & \ddots & \vdots\\
0 & 0 & \cdots & C_r
\end{pmatrix} \]
is a blocked diagonal matrix where each block $C_i$ is a companion matrix  
\[ 
C_i= \begin{pmatrix}
0 & 0 & \cdots & 0 & -\alpha_0\\
1 & 0 & \cdots & 0 & -\alpha_1\\
0 & 1 & \cdots & 0 & -\alpha_2\\
\vdots & \vdots& \dots & \vdots & \vdots\\
0 & 0 & \cdots & 1 & -\alpha_{k_i-1} 
\end{pmatrix}\]
of a polynomial $q(x)^m=\alpha_0 + \alpha_1 x + \cdots + \alpha_{k_i-1} x^{k_i-1} + x^{k_i}$ , where $q(x)$ is irreducible.
Since $q(x)$ is irreducible, it holds that either $q(x)^m=x^{k_i}$ and hence $\alpha_0=\alpha_1=\cdots=\alpha_{k_i-1}=0$ 
or $\alpha_0\neq 0$. \\[0.2cm]
Now, let $D\in \AID(K^n \rtimes_\varphi K)$. There exists an element $v\in K^n\rtimes_\varphi K$ such that 
$ D(t) = [t, v]$. By replacing $D$ with $D+\ad (v)$, we may assume that $D(t)=0$. \\[0.2cm]
For any vector $e \in K^n $, there exists a scalar $\alpha(e)\in K$, for which it holds that 
\[ D(e)=[ e, \alpha(e) t]. \]
Our aim is to show now that if both $e_{i,j}, e_{p,q}\not \in C_{\Lg}(t)$, then 
\[  
\alpha(e_{i,j})=\alpha(e_{p,q}).
\]
Since we assume that $e_{i,j}, e_{p,q}\not \in C_{\Lg}(t)$, it holds that 
\[ [t,e_{i,j}] = C_i e_{i,j} \neq 0 \mbox{ and } [t,e_{p,q}] = C_p e_{p,q}\neq 0.\]
Moreover by considering several cases we can see that $C_i e_{i,j}$ and $C_p e_{p,q}$ are linearly independent when 
$(i,j) \neq (p,q)$: \\[0.2cm]
{\it Case 1}, $i\neq p$: then $C_i e_{i,j}$ belongs to the span of $e_{i,1}, e_{i,2}, \ldots, e_{i,k_i}$, while 
$C_p e_{p,q}$ belongs to the span of $e_{p,1}, e_{p,2}, \ldots, e_{p,k_p}$, which shows that these vectors are linearly 
independent. \\[0.2cm]
{\it Case 2}, $i=p$: we may assume that $1 \leq j < q \leq k_i$. \\
In case $q<k_i$ we have that $C_i e_{i,j}= e_{i,j+1}$ and $C_i e_{i,q} = e_{i,q+1}$ which are clearly linearly independent.\\
In case $q=k_i$ we have that $C_i e_{i,k_i}= -\alpha_0 e_{i,1} -\alpha_1 e_{i,2} - \cdots - \alpha_{k_i-1} e_{i,k_i}$ 
with $\alpha_0 \neq 0$ (if $\alpha_0=0$, then also $\alpha_1=\cdots = \alpha_{k_i-1}=0$ and $e_{i,k_i}\in C_{\Lg}(t)$). 
Hence we obtain again that $C_i e_{i,j}$ and $C_i e_{q,k_i}$ are linearly independent. \\[0.2cm]
We find that 
\begin{equation}\label{semi1}
D(e_{i,j}+e_{p,q}) = [e_{i,j} + e_{p,q} , \alpha(e_{i,j} + e_{p,q}) t]=
- \alpha(e_{i,j} + e_{p,q}) C_i e_{i,j} - \alpha(e_{i,j} + e_{p,q}) C_p e_{p,q},
\end{equation}
while on the other hand we also have
\begin{equation}\label{semi2}
D(e_{i,j}) + D(e_{p,q}) = [e_{i,j}, \alpha(e_{i,j}) t] + [e_{p,q}, \alpha(e_{p,q}) t]=
-\alpha (e_{i,j} ) C_i e_{i,j} - \alpha(e_{p,q}) C_p e_{p,q}.\end{equation}
Since $\eqref{semi1}$ and $\eqref{semi2}$ must coincide and using the fact that  $C_i e_{i,j}$ and $C_p e_{p,q}$ are linearly 
independent, we finally find that 
\[ \alpha(e_{i,j}) =  \alpha(e_{i,j}+e_{p,q}) = \alpha(e_{p,q}).\]
Now, let $\alpha \in K$ be the fixed value such that $\alpha(e_{i,j}) =\alpha$ when $e_{i,j}\not\in C_{\Lg}(t)$, then we have that 
\[ D(e_{i,j})= [e_{i,j} , \alpha t] \mbox{ for all }1\leq i \leq r,\; 1 \leq j \leq k_i \mbox{ and also } 0= D(t) = [t , \alpha t].\]
It follows that $D$ coincides with $\ad (-\alpha t)$ on all basis vectors and hence $D=\ad (-\alpha t)\in \Inn(\Lg)$. 
\end{proof}

This result cannot be extended to Lie algebras $\Lg$ of the form $\Lg\cong K^n \rtimes K^2$. 
\begin{ex}
Let $n\geq 3$ and consider the Lie algebra $\Lg$  over $K$ with basis $e_1, e_2, \ldots, e_n, s,t $ and non-vanishing 
Lie brackets 
\begin{align*}
[s,e_i] & = e_{i+1},\; 1\leq i \leq n-1, \\
[t,e_i] & =e_{i+2},\; 1\leq i \leq n-2.
\end{align*}
Then we have $\Lg=K^n \rtimes K^2$. Let $D:\Lg\to \Lg$ be defined by 
$\alpha_1 e_1 + \cdots + \alpha_n e_n + \beta s + \gamma t \mapsto \gamma e_n$. Then $D$ is  a derivation. 
Define the map $\varphi_D: \Lg \to \Lg$ by
\[ 
\alpha_1 e_1 + \cdots + \alpha_n e_n + \beta s + \gamma t \mapsto 
\left\{ \begin{array}{ll}
\frac{\gamma}{\beta} e_{n-1} & \mbox{ if } \beta\neq 0\\
e_{n-2} & \mbox{ if } \beta=0.
\end{array} \right.\]
Then we have that for all $v\in \Lg$, $D(v)=[v,\varphi_D(v)]$ showing that $D\in \AID(\Lg)$. It is easy to see that 
$D\not\in \Inn(\Lg)$. Hence we have $\AID(\Lg)\neq \Inn(\Lg)$.
\end{ex}

This result can also not be generalized to Lie algebras of the form $\Lg\cong \Lf_{r,c} \rtimes K$ where 
$\Lf_{r,c}$ is a free nilpotent Lie algebra on $r$ generators and of class $c>1$.

\begin{ex} Let $\Lf_{3,2}$ be the free 2-step nilpotent Lie algebra on 3 generators, then $\Lf_{3,2}$ has a basis 
$x_1,x_2,x_3, y_1,y_2,y_3$ with non-trivial brackets
\[ [x_1,x_2]=y_1,\; [x_1,x_3]= y_2,\; [x_2,x_3]=y_3.\]
Now, add one more generator $t$ and one extra non trivial bracket 
\[ [t,x_1]= y_3,\]
to obtain a 7-dimensional Lie algebra $\Lg=\Lf_{3,2} \rtimes K$. Define $D:\Lg \to \Lg$ by
$$a_1 x_1 + a_2 x_2 + a_3 x_3 + b_1 y_1 + b_2 y_2 + b_3 y_3 + c t \mapsto a_1 (y_1+y_2).$$ Again, it is obvious 
that $D$ is a derivation of $\Lg$. Define $ \varphi_D: \Lg \to \Lg $ by
 \[ 
a_1 x_1 + a_2 x_2 + a_3 x_3 + b_1 y_1 + b_2 y_2 + b_3 y_3 + c t \mapsto 
\left\{ \begin{array}{ll}
x_2 + x_3 + \frac{a_2-a_3}{a_1}t & \mbox{ if } a_1 \neq 0\\
0 & \mbox{ if } a_1 =0.
\end{array} \right.
\]
Then $D(v) = [ v, \varphi_D(v)]$ for all $v\in \Lg$, showing that $D\in \AID(\Lg)$. It is easy to see that 
$D\not\in \Inn(\Lg)$, and so also in this case we have that $\AID(\Lg)\neq \Inn(\Lg)$.
\end{ex}

\section{Filiform nilpotent Lie algebras}

In this section we determine the almost inner derivations for the classes $L_n,Q_n,R_n,W_n$ of filiform nilpotent
Lie algebras discussed in \cite[Chapter 4]{GOZ} and for a family of characteristically nilpotent, 
filiform Lie algebras $\Lf_n$ for $n\ge 13$ introduced in \cite{BU35}. We always assume that $(e_1,\ldots ,e_n)$ is an 
adapted basis, which satisfies $[e_1,e_i]=e_{i+1}$ for all $2\le i\le n-1$. 

\begin{defi}
The Lie algebra $L_n$ for $n\ge 3$ is defined by the Lie brackets
\[
[e_1,e_i]=e_{i+1}, \quad 2\le i\le n-1.
\]
The Lie algebra $Q_n$ for $n\ge 6$ even is defined by the Lie brackets
\begin{align*}
[e_1,e_i] & = e_{i+1}, \quad 2\le i\le n-1,\\
[e_i,e_{n-i+1}] & = (-1)^{i+1}e_n, \quad 2\le i\le \frac{n}{2}.
\end{align*}
The Lie algebra $R_n$ for $n\ge 5$ is defined by the Lie brackets
\begin{align*}
[e_1,e_i] & = e_{i+1}, \quad 2\le i\le n-1,\\
[e_2,e_i] & = e_{i+2}, \quad 3\le i\le n-2.
\end{align*}
The Witt Lie algebra $W_n$  for $n\ge 5$ is defined by the Lie brackets
\begin{align*}
[e_1,e_j] & = e_{j+1}, \quad 2\le j\le n-1,\\[0.1cm]
[e_i,e_j] & = \frac{6(j-i)}{j(j-1)\binom{j+i-2}{i-2}} e_{i+j}, \quad 2\le i\le \frac{n-1}{2},\; i+1\le j\le n-i.
\end{align*}
\end{defi} 

The Witt algebra also has a basis $(f_1,\ldots ,f_n)$ with $[f_i,f_j]=(j-i)f_{i+j}$ for $1\le i+j\le n$, which is
not adapted. The derivation algebras of $L_n,Q_n,R_n,W_n$ have been determined in \cite{GOZ}. Since the algebras are filiform nilpotent, we have $\dim \Inn(\Lg)=n-1$ for all classes. The dimensions of $\Der(\Lg)$ are given as follows:
\begin{align*}
\dim \Der(L_n)& = 2n-1,\\
\dim \Der(Q_n)& =\frac{3n}{2},\\
\dim \Der(R_n)& =2n-3,\\
\dim \Der(W_n) & = n+3.
\end{align*}
We have shown that $\AID(L_n)=\Inn(L_n)$ in \cite[Proposition $7.2$]{BU55} and that 
\[
\AID(R_n)=\Inn(R_n)\oplus \langle E_{n,2}\rangle
\]
in \cite[Proposition $7.4$]{BU55}. Here $E_{ij}$ denotes the linear map which maps $e_j$ to $e_i$ and $e_k$ to $0$ for $k \neq j$. As a matrix, it has an entry $1$ at position $(i,j)$ and zero entries otherwise. Let $x=\sum_{i=1}^n x_ie_i\in Q_n$. Define linear maps in $\End(Q_n)$ by

\begin{align*}
t_0(x) & = x_2 e_n,\\
t_1(x) & = x_1e_1+x_1e_2+\sum_{i=3}^{n-1}(i-2)x_ie_i+(n-3)x_ne_n,\\
t_2(x) & = -x_1 e_2+\sum_{i=2}^{n-1} x_ie_i+2x_ne_n,\\
h_s(x) & = \sum_{i=2}^{n+1-2s}x_ie_{i-1+2s}, \quad 2\le s\le \frac{n}{2}-1.
\end{align*}

A computation shows that these linear maps are derivations of $Q_n$. We have the following result, see \cite{GOZ}.

\begin{prop}
Let $n\ge 6$ even. Then $\{\ad(e_1),\ldots ,\ad(e_{n-1}),t_0,t_1,t_2,h_2,h_3,\ldots ,h_{\frac{n}{2}-1}\}$ is a basis of 
$\Der(Q_n)$.
\end{prop}
Note that there is a mistake in the formulation and proof in \cite{GOZ}, since the map $t_0$ is not taken into account although it is a derivation. It corresponds with the map $d_{n-2}$ of the proof, which is not zero as is claimed there. It is easy to see then that every almost inner derivation of $Q_n$ is inner.

\begin{prop}
Let $n\ge 6$ even. Then $\AID(Q_n)=\Inn(Q_n)$.
\end{prop}

\begin{proof}
Take an arbitrary $\varphi \in \langle t_0, t_1, t_2, h_s \mid 2 \leq s \leq \frac{n}{2} -1 \rangle$, then there exist values $\alpha_0, \alpha_1, \alpha_2, \beta_s$ (with $2 \leq s \leq \frac{n}{2} - 1$) such that
\begin{equation*}
\varphi = \alpha_0 t_0 + \alpha_1 t_1 + \alpha_2 t_2 + \sum_{s = 2}^{\frac{n}{2} - 1} \beta_s h_s.
\end{equation*} Suppose that $\varphi \in \AID(Q_n)$. 
For $x=\sum_{i=1}^n x_ie_i\in Q_n$ we have
\[
[e_1+e_2,x]=(x_2-x_1)e_3+\sum_{i=4}^{n-1}x_{i-1}e_i.
\]
Since $\varphi(e_1 + e_2) = \alpha_0 e_n + \alpha_1 (e_1 + e_2) + \sum_{s = 2}^{\frac{n}{2} - 1} \beta_s e_{2s + 1}$, we must have that $\alpha_0 = \alpha_1 = 0$.\\
Moreover, $\varphi(e_2) = \alpha_2 e_2 + \sum_{s = 2}^{\frac{n}{2} - 1} \beta_s e_{2s + 1}$, but $[e_2,Q_n] = \langle e_3, e_n \rangle$, which means that $\alpha_2 = \beta_s = 0$ (for all $2 \leq s \leq \frac{n}{2} - 1$). Hence, the only almost inner derivation in $\langle t_0, t_1, t_2, h_s \mid 2 \leq s \leq \frac{n}{2} - 1 \rangle$ is $\varphi = 0$.
\end{proof}

For the Witt algebra $W_n$ define linear maps by
\begin{align*}
t_1(x) & = x_2 e_n,\\
t_2(x) & = x_2e_{n-1}+x_3e_n,\\
t_3(x) & = x_2 e_{n-2}+x_3e_{n-1}+x_4 e_n,\\
h(x) & = \sum_{i=1}^{n}ix_ie_{i}.
\end{align*}

We have the following result, see \cite{GOZ}.

\begin{prop}\label{4.4}
Let $n\ge 5$. Then $\{\ad(e_1),\ldots ,\ad(e_{n-1}),t_1,t_2,t_3,h\}$ is a basis of $\Der(W_n)$.
\end{prop}

From this we obtain the following result.

\begin{prop}\label{4.5}
Let $n\ge 9$. Then $\AID(W_n)=\Inn(W_n)\oplus \langle t_1,t_2,t_3 \rangle$. 
\end{prop}

\begin{proof}
The derivation $h$ is not almost inner, because $h(e_1)=e_1\not\in [e_1,W_n]$. We need to show that
$t_1,t_2,t_3$ are almost inner for all $n\ge 9$. Then the claim follows by Proposition $\ref{4.4}$.
Let us write $[e_i,e_j]=c_{i,j}e_{i+j}$, $2\le i<j\le n-i$ for the coefficients appearing in the Lie brackets of $W_n$.
Let $x=\sum_{i=1}^n x_ie_i\in W_n$. Define a map $\phi_{t_1}\in \End(W_n)$ by
\[
\phi_{t_1}(x)=\begin{cases} \frac{x_2}{x_1}e_{n-1}, \hspace{0.53cm} \text{ for } x_1\neq 0,\\[0.1cm]
\frac{1}{c_{2,n-2}}e_{n-2},  \text{ for } x_1= 0. 
\end{cases}
\]
Here $n-2\ge 3$ in $c_{2,n-2}$ since $n\ge 5$. We claim that $t_1(x)=[x,\phi_{t_1}(x)]$ for all $x\in W_n$, so that 
$t_1$ is almost inner. Indeed, for $x_1\neq 0$ we have $[x,\frac{x_2}{x_1}e_{n-1}]=x_2[e_1,e_{n-1}]=x_2e_n=t_1(x)$. For
$x_1=0$ we also have $[x, \frac{1}{c_{2,n-2}}e_{n-2}]=x_2\frac{1}{c_{2,n-2}}[e_2,e_{n-2}]=x_2e_n$. \\
For $n\ge 7$ we define a map $\phi_{t_2}\in \End(W_n)$ by
\[
\phi_{t_2}(x)=\begin{cases} \frac{x_2}{x_1}e_{n-2}+(\frac{x_3}{x_1}-\frac{c_{2,n-2}x_2^2}{x_1^2})e_{n-1}, 
\hspace{0.61cm} \text{ for } x_1\neq 0,\\[0.1cm]
\frac{1}{c_{2,n-3}}e_{n-3}+\frac{(c_{2,n-3}-c_{3,n-3})x_3}{c_{2,n-2}c_{2,n-3}x_2}e_{n-2}, \text{ for } x_1= 0,\, x_2\neq 0, \\[0.1cm]
\frac{1}{c_{3,n-3}}e_{n-3},  \hspace{3.85cm} \text{ for } x_1= 0,\, x_2=0.
\end{cases}
\]
This is well-defined for the $c_{i,j}$ since $n\ge 7$. We claim that $t_2(x)=[x,\phi_{t_2}(x)]$ for all $x\in W_n$, so that 
$t_2$ is almost inner for all $n\ge 7$. Indeed, for $x_1\neq 0$ we have
\begin{align*}
[x, \frac{x_2}{x_1}e_{n-2}+\left(\frac{x_3}{x_1}-\frac{c_{2,n-2}x_2^2}{x_1^2}\right)e_{n-1}] & = x_2[e_1,e_{n-2}]+\left( x_3-c_{2,n-2}
\frac{x_2^2}{x_1}\right)[e_1,e_{n-1}]+\frac{x_2^2}{x_1}[e_2,e_{n-2}] \\
 & = x_2e_{n-1}+x_3e_n-c_{2,n-2}\frac{x_2^2}{x_1}e_n+c_{2,n-2}\frac{x_2^2}{x_1}e_n\\
 & = t_2(x).
\end{align*}
For $x_1=0$ and $x_2\neq 0$ we have
\begin{align*}
[x,\frac{1}{c_{2,n-3}}e_{n-3}+\frac{(c_{2,n-3}-c_{3,n-3})x_3}{c_{2,n-2}c_{2,n-3}x_2}e_{n-2}] & = x_2\frac{1}{c_{2,n-3}}[e_2,e_{n-3}]+
\frac{c_{2,n-3}-c_{3,n-3}}{c_{2,n-2}c_{2,n-3}}x_3[e_2,e_{n-2}] \\
 & + \frac{x_3}{c_{2,n-3}}[e_3,e_{n-3}] \\
 & = x_2e_{n-1}+x_3\left( \frac{c_{2,n-3}-c_{3,n-3}}{c_{2,n-3}}+\frac{c_{3,n-3}}{c_{2,n-3}}\right) e_n \\
 & = t_2(x).
\end{align*}
Finally, for $x_1=x_2=0$ we have
\[
[x,\frac{1}{c_{3,n-3}}e_{n-3}]=\frac{x_3}{c_{3,n-3}}[e_3,e_{n-3}]=x_3e_n=t_2(x).
\]
For $n\ge 9$ we define a map $\phi_{t_3}\in \End(W_n)$ by 
\[
\phi_{t_3}(x)=\begin{cases} \rho_1(x), \text{ for } x_1\neq 0,\\
\rho_2(x), \text{ for } x_1=0,\, x_2\neq 0,\\
\rho_3(x), \text{ for } x_1=x_2=0,\, x_3\neq 0,\\
\rho_4(x), \text{ for } x_1=x_2=x_3=0,\\
\end{cases}
\]
with
\begin{align*}
\rho_1(x) & = \frac{x_2}{x_1}e_{n-3} +\left(\frac{x_3}{x_1}-\frac{c_{2,n-3}x_2^2}{x_1^2}\right)e_{n-2}+ \left(\frac{x_4}{x_1}-
\frac{(c_{2,n-2}+c_{3,n-3})x_2x_3}{x_1^2}+\frac{c_{2,n-2}c_{2,n-3}x_2^3}{x_1^3}\right)e_{n-1}, \\[0.1cm]
\rho_2(x) & = \frac{1}{c_{2,n-4}}e_{n-4} +\frac{(c_{2,n-4}-c_{3,n-4})x_3}{c_{2,n-3}c_{2,n-4}x_2}e_{n-3} \\[0.1cm]
 & + \left( \frac{(c_{2,n-4}-c_{4,n-4})x_4}{c_{2,n-2}c_{2,n-4}x_2}-\frac{(c_{2,n-4}-c_{3,n-4})c_{3,n-3}x_3^2}{c_{2,n-2}c_{2,n-3}c_{2,n-4}x_2^2}
\right)e_{n-2}, \\[0.1cm]
\rho_3(x) & = \frac{1}{c_{3,n-4}}e_{n-4}  + \frac{(c_{3,n-4}-c_{4,n-4})x_4}{c_{3,n-3}c_{3,n-4}x_3}e_{n-3}, \\[0.1cm]
\rho_4(x) & = \frac{1}{c_{4,n-4}}e_{n-4} .
\end{align*}
This is well-defined since $n\ge 9$. It is straightforward to see that $t_3(x)=[x,\phi_{t_3}(x)]$ for 
all $x\in W_n$, so that $t_3$ is almost inner for all $n\ge 9$. This finishes the proof.
\end{proof}

\begin{rem}
For $k=5,6$ we have $\AID(W_k)=\Inn(W_k)\oplus \langle t_1 \rangle$ and for $k=7,8$ we have $\AID(W_k)=\Inn(W_k)\oplus 
\langle t_1,t_2 \rangle$. 
\end{rem}

In \cite{BU35} we have introduced a family of filiform nilpotent Lie algebras $\Lf_n$ for $n\ge 13$. 
They are closely related to the Witt algebras $W_n$. The Lie brackets are defined as follows:

\begin{align*}
[e_1,e_i] & =e_{i+1}, \quad i=2,\dots ,n-1 \\[0.3cm]
[e_i,e_j] & =\sum_{r=1}^n\biggl(\;\sum_{\ell=0}^{\lfloor \frac{j-i-1}{2}\rfloor} (-1)^\ell
\binom{j-i-\ell-1}{\ell}\al_{i+\ell,\, r-j+i+2\ell+1}\biggr)e_r,
 \quad 2 \le i<j \le n,
\end{align*}

with parameters $\al_{k,s}$ for $1\le k,s\le n$, which are zero except for

\begin{align*}
\al_{\ell,2\ell+1} & = \frac{3}{\binom{\ell}{2}\binom{2\ell-1}{\ell-1}}, \quad 
\ell=2,3,\ldots ,\lfloor \textstyle{\frac{n-1}{2}}\rfloor , \\[0.3cm]
\al_{3,n-4} & =1, \\[0.3cm]
\al_{4,n-2} & = \frac{1}{7}+\frac{10}{21}\frac{(n-7)(n-8)}{(n-4)(n-5)}, \\[0.3cm]
\al_{4,n} & =
\begin{cases}
\frac{22105}{15246}, & \text{ if $n=13$, } \\
0  & \text{ if $n \ge 14$,}
\end{cases}
\end{align*}
and
\begin{align*}
\al_{5,n} & = \frac{1}{42}-\frac{70(n-8)}{11(n-2)(n-3)(n-4)(n-5)}+
\frac{25}{99}\frac{(n-6)(n-7)(n-8)}{(n-2)(n-3)(n-4)} \\[0.3cm]
          & \hspace{1.05cm} + \frac{5}{66} \frac{(n-5)(n-6)}{(n-2)(n-3)}-  
\frac{65}{1386} \frac{(n-7)(n-8)}{(n-4)(n-5)}.\\[0.3cm]
\end{align*}
For convenience consider the case $n=13$ separately. For $n\ge 14$ it is easy to see that 
the Lie brackets of $\Lf_n$ are given by
\begin{align*}
[e_1,e_j] & = e_{j+1}, \quad 2\le j\le n-1,\\[0.1cm]
[e_i,e_j] & = c_{i,j}e_{i+j}+d^n_{i,j}e_{i+j+n-11}, \quad 2\le i<j, \, i+j\le 11
\end{align*}
where the coefficients
\[
c_{i,j}=\frac{6(j-i)}{j(j-1)\binom{j+i-2}{i-2}}
\]
are the same as for $W_n$. This follows from the Pfaff-Saalsch\"utz formula, see \cite{BU35}.
It is also easy to see that $d_{2,5}^n=-1$ for all $n\ge 13$.
Define linear maps $t_1,t_2,t_3,h\in \End(\Lf_n)$ exactly like in the case of $W_n$. Note that
$h$ with $h(e_i)=ie_i$ is not a derivation of $\Lf_n$, because 
\[
[e_2,e_5]=c_{2,5}e_7+d_{2,5}^ne_{n-4}=\frac{9}{10}e_7-e_{n-4},
\]
so that 
\begin{align*}
h([e_2,e_5]) & = \frac{63}{10}e_7-(n-4)e_{n-4},\\[0.1cm]
[h(e_2),e_5] + [e_2,h(e_5)] & = 7[e_2,e_5]= \frac{63}{10}e_7-7e_{n-4}
\end{align*}
are different for all $n\ge 13$.

\begin{prop}
Let $n\ge 13$. Then $\{\ad(e_1),\ldots ,\ad(e_{n-1}),t_1,t_2,t_3\}$ is a basis of $\Der(\Lf_n)$. In particular,
all derivations of $\Lf_n$ are nilpotent.
\end{prop}

\begin{proof}
For $n\ge 14$ the proof goes exactly like the proof for the Witt algebra $W_n$, except that only
$t_1,t_2,t_3$ are derivations, but not $h$. For $n=13$ the claim can be verified explicitly.
\end{proof}

We obtain the remarkable result that all derivations of $\Lf_n$ for $n\ge 13$ are almost inner.
In particular, $t_1,t_2,t_3$ are almost inner, but not inner.

\begin{prop}
We have $\Der(\Lf_n)=\AID(\Lf_n)$ for all $n\ge 13$.
\end{prop}

\begin{proof}
Indeed, the derivations $t_1,t_2,t_3$ are almost inner. For $n\ge 14$ this follows in the same way as in the proof of
Proposition $\ref{4.5}$, using the functions $\phi_{t_1},\phi_{t_2},\phi_{t_3}$. These depend only on
the structure constants $c_{i,j}$, and all computations are also valid here. For $n=13$ the result can be verified
directly.
\end{proof}

\section{Lie algebras whose solvable radical is abelian}

The aim of this section is to show that $\AID(\Lg)=\Inn(\Lg)$ for any Lie algebra $\Lg$ whose solvable radical 
is abelian, over an algebraically closed field $K$ of characteristic zero. 
Fix such a Lie algebra $\Lg$ and let $\La$ denote the solvable radical 
of $\Lg$, which is abelian. Then we have $\Lg=\La\rtimes \Ls$, where $\Ls$ is a semisimple Lie algebra. 
Let $s\in \Ls$ and $a\in \La$. Denote by $s\cdot a=[s,a]$ the  $\Ls$-module  structure of $\La$. 
Then the Lie bracket in $\Lg$ is given by 
\[ [(a_1,s_1),(a_2,s_2)]= (s_1\cdot a_2 - s_2 \cdot a_1, [s_1,s_2]), \; \forall a_1,a_2\in \La,\; \forall s_1,s_2\in \Ls.\]
In the sequel we will use
\[ \End_\Ls(\La)=\{ \varphi: \La \to \La\;|\; \varphi \mbox{ is linear  and } \varphi(s\cdot a)=s\cdot \varphi(a),
\;\forall s\in \Ls,\;\forall a\in \La\}
\]
to denote the space of $\Ls$-endomorphisms of $\La$. For any $\varphi\in \End_\Ls(\La)$ we define 
\[
D_\varphi:\La\rtimes \Ls \to \La\rtimes \Ls: (a,s) \mapsto (\varphi(a), 0).
\]
\begin{lem}
We have $D_\varphi \in \Der(\La \rtimes \Ls)$. 
\end{lem}
\begin{proof} Let $a_1,a_2\in \La$ and $s_1,s_2\in \Ls$.
On the one hand, we have that 
\begin{equation}\label{LHSD}
D_\varphi([(a_1,s_1),(a_2,s_2)])=D_\varphi(s_1\cdot a_2 - s_2 \cdot a_1, [s_1,s_2])= (\varphi(s_1\cdot a_2 - s_2 \cdot a_1),0),
\end{equation}
while on the other hand
\begin{equation}\label{RHSD}
[D_\varphi(a_1,s_1), (a_2,s_2) ] + [ (a_1,s_1) , D_\varphi(a_2,s_2)] = (-s_2\cdot \varphi(a_1) + s_1 \cdot \varphi(a_2),0).
\end{equation}
Since $\varphi\in \End_\Ls(\La)$, it holds that \eqref{LHSD} coincides with \eqref{RHSD}, which shows that 
$D_\varphi\in \Der(\La\rtimes \Ls)$.
\end{proof}

\begin{prop} Let $\LD= \{ D_\varphi\;|\; \varphi \in \End_\Ls (\La) \}$. Then as 
vector spaces we have that 
\[ \Der(\La\rtimes  \Ls) = \Inn(\La\rtimes \Ls) \oplus \LD.\]
\end{prop}


\begin{proof}
It is easy to see that $\Inn(\La\rtimes \Ls) \cap \LD = 0$, so we have to show that 
$\Der(\La\rtimes  \Ls) = \Inn(\La\rtimes \Ls) +\LD$. Now, consider any $D\in \Der(\La\rtimes \Ls)$. 
The derivation $D$ induces a derivation on $\Ls$, which is an inner 
derivation, since $\Ls$ is semisimple. So, after changing $D$ up to an inner derivation, we may assume that 
$D$ induces the zero map on $\Ls$. It follows that there exists a linear map $f:\Ls \to \La$ such that 
$D(0,s) = (f(s),0)$ for all $s\in \Ls$. As $D$ is a derivation we have that 
\[ D(0,[s_1,s_2]) = D[ (0,s_1), (0,s_2) ] = [D(0,s_1), (0,s_2)]+[(0,s_1),D(0,s_2)].\]
And so
\[ (f([s_1,s_2]),0)= [(f(s_1) ,0),(0,s_2)]+ [(0,s_1),(f(s_2) ,0)], 
\]
which gives 
\[ f([s_1,s_2])= s_1\cdot f(s_2) - s_2 \cdot f(s_1). \]
This means that $f\in Z^1(\Ls,\La)$ is a 1--cocycle. As $\Ls$ is semisimple, we have that $H^1(\Ls,\La)=0$ and so there 
exists an element $a_0\in \La$ such that $f(s)= s\cdot a_0$ for all $s\in \Ls$. Now
\[ \forall s\in \Ls:\;(D+\ad ((a_0,0))(0,s) = (f(s),0)+[(a_0,0),(0,s)]= (0,0).\]
This means that after changing $D$ again with an inner derivation we will assume that $D(\Ls)=0$ and hence there is a 
linear map $ \varphi:\La \to \La$ such that 
\[ \forall a\in \La,\;\forall s \in \Ls:\; D(a,s)= (\varphi(a),0).\]
Now, using the fact that $D$ is a derivation, we must have that 
\[ D[(a,0),(0,s)] = [D(a,0), (0,s)] +  [(a,0),D(0,s)]\]
which implies that 
\[ D(-s\cdot a, 0) = [(\varphi(a),0), (0,s) ]\Rightarrow \varphi(s\cdot a) = s\cdot \varphi(a).\]
This shows that after changing $D$ up to an inner derivation we have that $D=D_\varphi\in \LD$ which finishes the proof.
\end{proof}
We are now ready to prove the main result of this section.

\begin{thm}
Let $\Lg$ be a Lie algebra over an algebraically closed field $K$ of characteristic zero whose solvable radical is abelian. 
Then $\AID(\Lg)=\Inn(\Lg)$.
\end{thm}

\begin{proof}
As before, we can write $\Lg=\La\rtimes \Ls$, with $\La$ abelian and $\Ls$ semisimple. We have shown that 
$\Der(\Lg)=\Inn(\Lg) \oplus \LD$.
In order to prove the result, we have to show that  if a  derivation $D\in \LD$  is not the zero map, then $D$ is not 
an almost inner derivation. 
So consider a nonzero $D\in \LD$, then $D=D_\varphi$ for some nonzero $\varphi\in \End_\Ls(\La)$. 
Let $V=\varphi(\La)$ be the image of $\varphi$. Then $V$ is a nonzero $\Ls$--submodule of $\La$. 
The Lie algebra $\Ls$ contains a nilpotent element $s_0$ such that $C_\Ls(s_0)$ consists entirely of nilpotent elements,
see for example \cite[section $35$]{TAY}. In particular $C_\Ls(s_0)$ is also nilpotent as a Lie algebra. 
Consider the map $\psi:\Ls \to \End(V):s \mapsto \psi(s)$, 
where $\psi(s)(v)=s\cdot v$. Then $\psi$ is a representation of Lie algebras and since $\Ls$ is semisimple $\psi$ maps 
nilpotent elements to nilpotent elements. 
It follows that $\psi(C_\Ls(s_0))$ consists of nilpotent endomorphisms and in particular, it follows that 
$\psi(C_\Ls(s_0))(V)= C_\Ls(s_0) \cdot V$ is strictly contained in $V$. Let $v_0\in V\backslash (C_\Ls(s_0)\cdot V)$ and pick an  
$a_0\in \La$ with $\varphi(a_0)=v_0$. Note that since $\Ls$ is semisimple, we can find a complementary 
$\Ls$--submodule $W$ of $V$ in $\La$ such that  $\La$ decomposes as a direct sum $\La=V\oplus W$ of $\Ls$--modules.
In particular, we also find that $v_0\in \La \setminus C_\Ls(s_0)\cdot \La$. \\
We claim that $D_\varphi(a_0,s_0) \not\in [(a_0,s_0), \Lg]$, which shows that $D_\varphi$ is not an almost 
inner derivation. Indeed, assume that 
\[ D_\varphi(a_0, s_0)= [(a_0,s_0), (a,s)]\mbox{ for some }a\in \La\mbox{ and some }s\in \Ls.\]
Then we have that 
\[
(\varphi(a_0),0 )  = [(a_0,s_0), (a,s)] = (s_0\cdot a - s \cdot a_0, [s_0,s]).  
\]
This shows that $[s_0,s]=0 $ and so $s\in C_\Ls(s_0)$. However, this now implies that 
$\varphi(a_0)=v_0= s_0\cdot a - s \cdot a_0\in C_\Ls(s_0)\cdot \La$ which is a contradiction with the fact that we 
have chosen $v_0$ such that $v_0\in \La \setminus C_\Ls(s_0)\cdot \La$.
\end{proof}

\section{Change of base field}

Consider a field extension $K$ of $k$. Let $\Lg_k$ be a Lie algebra over $k$ and denote by $\Lg_K = K \otimes_k \Lg_k$ 
the corresponding Lie algebra over $K$. This is an extension of scalars. We will assume that $k$ has characteristic zero and 
that the field extension is of finite degree $[K:k] = n$. The primitive element theorem ensures that $K = k(s)$ for some $s \in K$. 
Then $\CB = \{1,s,s^2, \ldots,s^{n-1}\}$ is a vector space basis of $K$ over $k$. It follows that
\begin{equation}
\label{eq: vector spaces}
	\Lg_K = \Lg_k \oplus s \Lg_k \oplus \ldots \oplus s^{n - 1} \Lg_k
\end{equation} holds as vector spaces over $k$. The typical example is $K = \C$ and $k = \R$, where $\{1,s\} = \{1,i\}$ and 
$\Lg_\C = \Lg_\R \oplus i\Lg_\R$.
We can also consider $\Lg_K$ as a Lie algebra over $k$. We will denote this Lie algebra with $\Lg_k'$. Note that, as sets, we have $\Lg_k' = \Lg_K$. Finally, $\Lg_K' := K \otimes_k \Lg_k'$ is again a Lie algebra over $K$. \\
Let us for the moment consider the  special situation when $[K:k] = 2$. In that case we have that $K = k(\alpha)$ for some $\alpha \in K \setminus k$ 
with $d := \alpha^2 \in k$. Hence, we can write $K = k(\sqrt{d})$. Now let $\Lg_k$ be a Lie algebra over $k$ of dimension $r$ 
with basis $\{e_1,e_2,\ldots,e_r\}$ and structure constants $c_{ij}^p$, so $[e_i,e_j] = \sum_{p = 1}^r c_{ij}^p e_p$. Then 
$\Lg_K = K \otimes_k \Lg_k$ has the same structure constants and the same basis. Further, $\Lg_k'$ has basis 
$\{e_1,e_2,\ldots,e_r,\alpha e_1,\alpha e_2,\ldots, \alpha e_r\}$. Denote $f_i := \alpha e_i$ for all $1 \leq i \leq r$, then 
the structure constants are
\begin{equation*}
[e_i,e_j] = \sum_{p = 1}^r c_{ij}^p e_p, \qquad [e_i,f_j] = \sum_{p = 1}^r c_{ij}^p f_p, \qquad  
[f_i,f_j] = d \sum_{p = 1}^r c_{ij}^p e_p = d[e_i,e_j].
\end{equation*} The Lie algebra $\Lg_K'$ has the same basis and structure constants as $\Lg_k'$.

\begin{lem}
\label{lem: extension of dimension 2}
For $[K:k] = 2$ we have $\Lg_K' \cong \Lg_K \oplus \Lg_K$.
\end{lem}
\begin{proof}
The Lie algebra $\Lg_K'$ has basis $\{e_1,e_2,\ldots,e_r,f_1,f_2,\ldots,f_r\}$ and structure constants as above. We take for 
$\Lg_K \oplus \Lg_K$ a basis $\{a_1,a_2,\ldots,a_r,b_1,b_2,\ldots,b_r\}$ with
\begin{equation*}
	[a_i,a_j] = \sum_{p = 1}^r c_{ij}^p a_p, \qquad [b_i,b_j] = \sum_{p = 1}^r c_{ij}^p b_p, \qquad [a_i,b_j] = 0.
\end{equation*} Let $\varphi: \Lg_K' \to \Lg_K \oplus \Lg_K$ be the linear map with $\varphi(e_i) = a_i + b_i$ and 
$\varphi(f_i) = \alpha a_i - \alpha b_i$ for all $1 \leq i \leq r$. Then $\varphi$ is an isomorphism of vector spaces. 
Moreover, $\varphi$ is a Lie algebra morphism. Indeed, take $1 \leq i,j \leq r$ arbitrarily, then we have that
\begin{equation*}
[\varphi(e_i),\varphi(e_j)] = [a_i + b_i,a_j + b_j] = [a_i,a_j] + [b_i,b_j] = \sum_{p = 1}^r c_{ij}^p (a_p + b_p) = \varphi([e_i,e_j]).
\end{equation*} Furthermore, also 
\begin{equation*}
	[\varphi(e_i),\varphi(f_j)] = [a_i + b_i, \alpha a_j - \alpha b_j] = \alpha \sum_{p = 1}^r c_{ij}^p (a_p - b_p) 
= \varphi\left(\sum_{p = 1}^r c_{ij}^p f_p \right) = \varphi([e_i,f_j])
\end{equation*} is satisfied. Finally, we find that 
\begin{align*}
	[\varphi(f_i),\varphi(f_j)] & = [\alpha(a_i - b_i),\alpha(a_j - b_j)] = \alpha^2 [a_i,a_j] + \alpha^2 [b_i,b_j] \\
	& = d \left( [a_i,a_j] + [b_i,b_j] \right) = \varphi(d[e_i,e_j]) \\
	& = \varphi([f_i,f_j]).
\end{align*}
\end{proof}
\begin{rem}
Let $k \subseteq K$ be a Galois extension and $\Lg_K$ a Lie algebra over $K$ with underlying Lie algebra $\Lg_k'$. A more general result about the structure of $\Lg_K'$ can be found in \cite{DER}.
\end{rem}
Now, we return again to the general situation where $[K:k]=n \geq 2$.
Suppose that $D \in \Der(\Lg_k)$, then we can consider $D_K = 1_K \otimes_k D$ where $D_K: \Lg_K \to \Lg_K$ is the 
$K$-linear map such that ${D_K}_{\mid \Lg_k} = D$. This means that $D_K \in \Der(\Lg_K)$. Conversely, if $D \in \Der(\Lg_K)$ and 
$D(\Lg_k) \subseteq \Lg_k$, then $D_{\mid \Lg_k} \in \Der(\Lg_k)$. 

\begin{rem}
These two ``procedures" are inverses of each other. Indeed, for $D \in \Der(\Lg_k)$, we have that ${D_K}_{\mid \Lg_k} = D$. 
Moreover, for $D \in \Der(\Lg_K)$ with $D(\Lg_k)\subseteq \Lg_k$ we have $\left(D_{\mid \Lg_k}\right)_K = D$.
\end{rem}

\begin{lem}
We have $\Der(\Lg_K) = K \otimes_k \Der(\Lg_k)$.
\end{lem}
\begin{proof}
We already mentioned that any derivation $D \in \Der(\Lg_k)$ can be viewed as a derivation $D_K \in \Der(\Lg_K)$. From this, 
the conclusion $K \otimes_k \Der(\Lg_k) \subseteq \Der(\Lg_K)$ is clear. We now show the other inclusion. 
Let $D \in \Der(\Lg_K)$. We can write
	\begin{equation*}
		D_{\mid \Lg_k} = D_1 + s D_2 + \ldots + s^{n - 1} D_n,
	\end{equation*} where $D_i: \Lg_k \to \Lg_k$ is a derivation for all $1 \leq i \leq n$. Take 
	\begin{equation*}
		D' = 1_K \otimes_k D_1 + s \otimes_k D_2 + \ldots + s^{n - 1} \otimes_k D_n,
	\end{equation*}
	then we find that $D' \in K \otimes_k \Der(\Lg_k)$. Since also $D_{\mid \Lg_k} = {D'}_{\mid \Lg_k}$ holds, this 
implies that $D = D' \in K \otimes_k \Der(\Lg_k)$. 
\end{proof} 
Hence, there is a nice correspondence between the derivations of $\Lg_k$ and $\Lg_K$. 
\begin{lem}
Let $D \in \Der(\Lg_k)$. If $D_K \in \AID(\Lg_K)$, then also $D \in \AID(\Lg_k)$.
\end{lem}
\begin{proof}
Let $\CB = \{1, s,\ldots, s^{n - 1}\}$ be a basis of $K$ over $k$ and $D: \Lg_k \to \Lg_k$ be a derivation. Assume that 
$D_K \in \AID(\Lg_K)$, then there exists a map $\varphi: \Lg_K \to \Lg_K$ such that $D_K(x) = [x,\varphi(x)]$ holds 
for all $x \in \Lg_K$. We can write
\begin{equation*}
\varphi(x) := \varphi_1(x) + s \varphi_2(x) + \ldots + s^{n - 1}\varphi_n(x),
\end{equation*} where $\varphi_i: \Lg_K \to \Lg_k$ for all $1 \leq i \leq n$. 
Now take an arbitrary $x \in \Lg_k$. Then we obtain
\begin{equation*}
D(x) = [x,\varphi_1(x)] + s [x,\varphi_2(x)] + \ldots + s^{n - 1} [x,\varphi_n(x)].
\end{equation*} 
Since $D(x) \in \Lg_k$, it follows from equation (\ref{eq: vector spaces}) that for all $x \in \Lg_k$,  
\begin{equation*}
D(x) = [x,\varphi_1(x)] \in \Lg_k
\end{equation*} and $[x,\varphi_i(x)] = 0$ for all $2 \leq i \leq n$. Hence, this means that $D \in \AID(\Lg_k)$. 
\end{proof} 
Note that the converse of this result does not hold in general since there exist examples for which $D \in \AID(\Lg_k)$, but 
$D_K \notin \AID(\Lg_K)$, see Example \ref{ex: explanation of the concepts}.

\begin{prop}
If $\AID(\Lg_K) \neq \Inn(\Lg_K)$, then also $\AID(\Lg_k) \neq \Inn(\Lg_k)$.
\end{prop} 

\begin{proof}
Denote as before $\CB = \{1, s,\ldots, s^{n - 1}\}$ for a basis of $K$ over $k$. Let $D \in \AID(\Lg_K)$, 
$D \notin \Inn(\Lg_K)$, 
then there exists a map $\varphi: \Lg_K \to \Lg_K$ such that $D(x) = [x,\varphi(x)]$ for all $x \in \Lg_K$. Furthermore, there 
are maps $\varphi_i: \Lg_K \to \Lg_k$ (for all $1 \leq i \leq n$) such that 
\begin{equation*}
\varphi = \varphi_1 + s \varphi_2 + \ldots + s^{n - 1} \varphi_n.
\end{equation*}
Now define for each $1 \leq i \leq n$ the map
\begin{equation*}
D_i: \Lg_k \to \Lg_k: x \mapsto [x,\varphi_i(x)].
\end{equation*}
We claim that each $D_i$ is a derivation (and thus an almost inner derivation).\\
Let $x,y \in \Lg_k$, then
\begin{align*}
D([x,y]) & = \big[[x,y],\varphi_1([x,y]) + s \varphi_2([x,y]) + \ldots + s^{n - 1} \varphi_n([x,y]) \big] \\
& = D_1([x,y]) + s D_2([x,y]) + \ldots + s^{n - 1} D_n([x,y]).
\end{align*} On the other hand, we have 
\begin{align*}
& [D(x),y] + [x,D(y)]\\
& \quad = \big[ [x, \varphi_1(x) + s \varphi_2(x) + \ldots + s^{n - 1} \varphi_n(x)],y\big] + \big[ x, [y,\varphi_1(y) + 
s \varphi_2(y) + \ldots + s^{n - 1} \varphi_n(y) ] \big] \\
& \quad = [D_1(x) + s D_2(x) + \ldots + s^{n - 1} D_n(x),y] + [x,D_1(y) + s D_2(y) + \ldots + s^{n - 1} D_n(y)] \\
& \quad = [D_1(x),y] + [x,D_1(y)] + s \big( [D_2(x),y] + [x,D_2(y)]\big) + \ldots + s^{n - 1} \big([D_n(x),y] + [x,D_n(y)]\big).
\end{align*} 
Since $D$ is a derivation the above equations imply that $D_i([x,y])=[D_i(x),y] + [x,D_i(y)]$, hence $D_i \in \Der(\Lg_k)$ 
for all $1 \leq i \leq n$. \\
Moreover, we claim that there exists at least one $1 \leq i \leq n$ for which $D_i \notin \Inn(\Lg_k)$. Suppose on the 
contrary that $D_i \in \Inn(\Lg_k)$ for all $1 \leq i \leq n$. Then there exist $\alpha_i \in \Lg_k$ such that 
$D_i(x) = [x,\alpha_i]$ for all $x \in \Lg_k$. Denote $\alpha := \alpha_1 + s \alpha_2 + \ldots + s^{n - 1} \alpha_n \in K$. 
This means that
\begin{align*}
D(x) & = [x,\alpha_1] + s [x, \alpha_2] + \ldots + s^{n - 1} [x,\alpha_n] \\
& = [x,\alpha_1 + s \alpha_2 + \ldots + s^{n - 1} \alpha_n]\\
& = [x,\alpha]
\end{align*} for all $x \in \Lg_k$. Now consider $-\ad(\alpha) \in \Der(\Lg_K)$, then $D_{\mid \Lg_k} = -\ad(\alpha)_{\mid \Lg_k}$. 
Since two derivations of $\Lg_K$ are equal when they agree on $\Lg_k$, this implies that $D = - \ad(\alpha)$ is inner. 
This contradiction shows that for at least one $1 \leq i \leq n$, we have $D_i \in \AID(\Lg_k) \setminus \Inn(\Lg_k)$.
\end{proof}

This proposition means that if the Lie algebra over the bigger field $\Lg_K$ admits a non-trivial almost inner derivation, 
then also the Lie algebra over the smaller field $\Lg_k$. The converse does not hold in general, see again
Example \ref{ex: explanation of the concepts}.

\section{Constructing new almost inner derivations}

We keep using the same notations as in the previous section.
In this section, we will show how to find new almost inner derivations of the Lie algebra $\Lg_k'$, determined by 
$\AID(\Lg_K)$. Remember that $\Lg_k' = \Lg_K$ as a set, but now viewed as a Lie algebra over $k$. \\
Define the set 
\begin{equation*}
\CC(\Lg_K) := \{D \in \AID(\Lg_K) \mid D(\Lg_K) \text{ is one-dimensional and } D(\Lg_K) \subseteq Z(\Lg_K) \}.
\end{equation*}
We will show how to construct, starting from a fixed element $D \in\CC(\Lg_K)$ a collection of almost inner derivations 
of $\Lg_k'$ which are not inner, even when $D$ itself is an inner derivation of $\Lg_K$. \\
So fix some $D \in \CC(\Lg_K)$. Since $D(\Lg_K)$ is one-dimensional,  $\ker(D)$ is of codimension $1$. Hence, 
$\Lg_K = \langle y \rangle + \ker(D)$ for some $y \in \Lg_K \setminus \ker(D)$.  We also fix a choice of $y$ 
and let  $0 \neq z = D(y)$. Any element of $\Lg_K$ can be written as $ay + c$, where 
$a \in K$ and $c \in \ker(D)$. Denote again $\CB = \{1, s,\ldots, s^{n - 1}\}$ for a basis of $K$ over $k$, then any 
element $a \in K$ can be uniquely written as $a = a_1 + a_2 s + \ldots a_n s^{n - 1}$ with $a_i \in k$ for all 
$1 \leq i \leq n$. 
We use the notation $c_i(a) := a_i$ to denote the $i$-th coordinate of $a$ with respect to the basis $\CB$. 
We now have that $D: \Lg_K \to \Lg_K: ay + c \mapsto az$. Since $D \in \AID(\Lg_K)$, there exists a map 
$\varphi_D: \Lg_K \to \Lg_K$ such that 
\begin{equation*}
D(ay + c) = [ay + c, \varphi_D(ay + c)].
\end{equation*} 
Associated to $D$ we introduce $n$ new $k$-linear maps of $\Lg_k'$, namely
\begin{equation*}
D_i: \Lg_k' \to \Lg_k': ay + c \mapsto c_i(a) s^{i - 1} z,
\end{equation*} where $1 \leq i \leq n$. Note that $D=D_1+D_2+\cdots+ D_n$.
We remark here that the maps $D_i$ do depend on the choice of $y$, so in what follows we always assume that for a given 
$D$, a fixed $y$ outside of $\ker(D)$ has been chosen. We can now multiply the above maps 
with powers of $s$ to get a total of $n^2$ $k$-linear maps $s^{j-1} D_i : \Lg_k' \to \Lg_k'$, with $1 \leq i,j \leq n$.

\begin{lem} \label{lem: s^j D_i is AID}
For any $D \in \CC(\Lg_K)$ we have $s^{j - 1} D_i \in \AID(\Lg_k')$ for all $1 \leq i,j \leq n$. 
\end{lem}

\begin{proof}
First note that $[\Lg_K, \Lg_K] \subseteq \ker(D)$. Indeed, let $x,y \in \Lg_K$. Because $D(\Lg_K) \subseteq Z(\Lg_K)$ 
by definition, we have that $[D(x),y] + [x,D(y)] = 0$ and 
\begin{equation*}
D([x,y]) = [D(x),y] + [x,D(y)] = 0.
\end{equation*} This last equation implies that $D_i([x,y]) = 0$ and hence, $D_i$ is a derivation. Since $D$ is almost inner, 
it is determined by a map $\varphi_D: \Lg_K \to \Lg_K$. Define
\begin{equation*}
\varphi_{D_i}: \Lg_k' \to \Lg_k': ay + c \mapsto 
\begin{cases}
\frac{c_i(a)s^{i-1}}{a} \varphi_D(ay + c) & \text{ if } a \neq 0;\\
0 & \text{ if } a = 0.
\end{cases}
\end{equation*} 
For $a = 0$, we have that 
\begin{equation*}
D_i(ay + c) = 0 = [ay + c, \varphi_{D_i}(ay + c)].
\end{equation*} When $a \neq 0$, it follows that 
\begin{align*}
[ay + c, \varphi_{D_i}(ay + c)] & = \left[ay + c, \frac{c_i(a)s^{i-1}}{a} \varphi_D(ay + c)\right] \\[0.1cm]
   & = \frac{c_i(a)s^{i-1}}{a} D(ay + c) \\[0.1cm]
   & = c_i(a) s^{i-1} z \\
   & = D_i(ay + c).
\end{align*}
This shows that $D_i \in \AID(\Lg_k')$ for all $1 \leq i \leq n$. Moreover, for each $j \in \{1,2,\ldots,n\}$, the map  
$s^{j - 1} D_i: \Lg_k' \to \Lg_k'$  is an almost inner derivation, determined by the map 
$\varphi_{s^{j - 1} D_i} := s^{j - 1} \varphi_{D_i}$.  
\end{proof}

In this way, each $D \in \CC(\Lg_K)$ gives rise to $n^2$ almost inner derivations $s^{j - 1} D_i$ of $\Lg_k'$, 
where $1 \leq i,j \leq n$.

\begin{prop}\label{lem: s^j D_i is not inner}
Suppose that $D \in \CC(\Lg_K)$ and define $A := \langle s^{j - 1} D_i: \Lg_k' \to \Lg_k' \mid 1 \leq i, j \leq n \rangle$,  
the $k$--vector space spanned by the maps $s^{j - 1} D_i$. Then $\dim A= n^2$ and 

\begin{equation*}
A \cap \Inn(\Lg_k') =
\begin{cases}
\langle s^{j - 1} D: \Lg_k' \to \Lg_k' \mid 1 \leq j \leq n\rangle & \text{ if } D \in \Inn(\Lg_K),\\
\{0\} & \text{ if }  D \notin \Inn(\Lg_K).
\end{cases}
\end{equation*}
Hence we have $\dim (A \cap \Inn(\Lg_k'))=n$ or $\dim (A \cap \Inn(\Lg_k'))=0$.
\end{prop}

\begin{proof}
We will first show that the maps $s^{j-i} D_i$ are $k$--linearly independent. So assume that $\alpha_{i,j}\in k$ and that 
\[ \sum_{i=1}^n \sum_{j=1}^n \alpha_{i,j} s^{j-1} D_i =0 .\]
Applying the above to $s^{k-1} y$ for some $k\in \{1,\ldots ,n\}$ gives $\displaystyle \left( \sum_{j=1}^n \alpha_{k,j} 
s^{j-1}\right) s^{k-1} z =0$. As $s^{k-1} z \neq 0$, it follows that  $\displaystyle \sum_{j=1}^n \alpha_{k,j} s^{j-1}=0$ 
and since $\CB = \{1,s,s^2, \ldots,s^{n-1}\} $ is a basis of $K$ over $k$, it follows that all coefficients $\alpha_{k,j}=0$, 
showing that the maps $s^{j-i} D_i$ are $k$--linearly independent. \\[0.2cm]
Now assume that $E \in A \cap \Inn(\Lg_k')$. Let $E=\sum_{i=1}^n\sum_{j=1}^n \alpha_{i,j} s^{j-1} D_i$ and assume that $E=\ad (x)$ 
for some $x\in \Lg_k'$. Note that $E$ is also $K$-linear, since $E$ can also be seen as an inner derivation of $\Lg_K$. 
As above we have 
\[
E(s^{k-1} y) = \left( \sum_{j=1}^n \alpha_{k,j} s^{j-1}\right) s^{k-1} z 
\]
for every $k=1,2,\ldots , n$. On the other hand, since $E$ is also 
$K$--linear, it must hold that $E(s^{k-1} y) = s^{k-1} E(y)$ and therefore we get the equality 
\[ 
\sum_{j=1}^n \alpha_{k,j} s^{j-1}=  \sum_{j=1}^n \alpha_{1,j} s^{j-1}.
\]
It follows that for all $j$ we have that $\alpha_{1,j}=\alpha_{2,j}= \cdots = \alpha_{k,j}$ and we let $\beta_j=\alpha_{1,j}$ be 
this common value. Hence $E= \sum_{i=1}^n (\sum_{j=1}^n \beta_j s^{j-1}) D_i= 
\beta (D_1 +D_2 + \cdots + D_n)=
\beta D$, where $\beta= \sum_{j=1}^{n} \beta_j s^{j-1} \in K$. So $E= \beta D$ for some $\beta \in K$ and therefore 
$E\in \Inn(\Lg_K)$ if and only if $D \in \Inn (\Lg_K)$ and if this is the case, the above shows that 
$E \in \langle s^{j-1} D\mid 1\leq j \leq n \rangle$ which finishes the proof, since as sets $\Inn(\Lg_K)=\Inn(\Lg_k')$.
\end{proof}

In many cases, the above proposition allows us to construct Lie algebras over a non algebraically closed field $k$ with many 
almost inner derivations which are not inner. As an example of this we have the following result.

\begin{cor} Let $K$ be a field extension of a field $k$, with $[K:k]= n\geq 2$. Using the notation from above, 
assume that $\Lg_k$ is a $c$-step nilpotent Lie algebra for $c\geq2$ with $\dim( \Lg_k^c)=1$, then 
$\dim (\AID(\Lg_k')) - \dim(\Inn(\Lg_k')) \geq n^2 -n >0$.
\end{cor}

\begin{proof}
Let $v \in \Lg_k^{c-1}$ be an element with $[v,\Lg_k]\neq 0$. Such a $v$ exists, since we assume that 
$\Lg_k$ is $c$--step nilpotent. Moreover $\Lg_k^c = {\rm Im}(\ad(v)) \subseteq Z(\Lg_k)$. 
It follows that $D=\ad(v) \in \CC(\Lg_K)$. By Proposition~\ref{lem: s^j D_i is not inner} we know that 
$A=\langle s^{j-1} D_i \mid 1\leq i,j \leq n \rangle$ is an $n^2$-dimensional subspace of $\AID(\Lg_k')$ intersecting 
$\Inn(\Lg_k')$ in a $n$-dimensional space, which proves the corollary.
\end{proof}

Note that the above corollary can for example be applied when $\Lg_k$ is a filiform Lie algebra. 
We finish this paper with an example where $k=\R$ and $K=\C$, i.e., with $n=2$. Then we can take $s=i$ with $i^2=-1$.
\begin{ex} \label{ex: explanation of the concepts}
Consider the real Heisenberg Lie algebra $\Lg_\R$, then $\Lg_\C $ is the complex Heisenberg Lie algebra and both of 
these algebras can be described via a basis $\{e_1,e_2,e_3\}$, where the non-zero Lie brackets are given by 
$[e_1,e_2] = e_3$. Then we have
\begin{align*}
\AID(\Lg_\R) & = \Inn(\Lg_\R), \\
\AID(\Lg_\C) & = \Inn(\Lg_\C).
\end{align*}
It follows from Lemma $\ref{lem: extension of dimension 2}$ that 
\begin{equation*}
\AID(\Lg_\C') \cong \AID(\Lg_\C \oplus \Lg_\C) =  \AID(\Lg_\C) \oplus \AID(\Lg_\C) = 
\Inn(\Lg_\C) \oplus\Inn (\Lg_\C) \cong \Inn(\Lg_\C').
\end{equation*}
Let $D = \ad(e_1)$ and $E = \ad(e_2)$ be inner derivations of $\Lg_\C$, then both $D,E \in \CC(\Lg_\C)$.\\
The Lie algebra $\Lg_\R'$ has a basis $\{ e_1, e_2, e_3,f_1= i e_1, f_2= i e_2, f_3= i e_3\} $ and non-zero brackets
\[ [e_1, e_2]= e_3, \; [e_1,f_2]= f_3,\; [f_1, e_2] = f_3,\; [f_1, f_2] = - e_3.\]
We can now consider the $\R$--linear maps $D_1, \; iD_1,\; E_1,\; i E_1$ as defined above and these satisfy:
\begin{align*}
 D_1(e_2) & = e_3,\; D_1(e_1)=D_1(f_1) = D_1 (f_2) = D_1(e_3) = D_1( f_3) =0,\\
i D_1(e_2) & = f_3,\; D_1(e_1)=D_1(f_1) = D_1 (f_2) = D_1(e_3) = D_1( f_3) =0, \\
E_1(e_1) & = - e_3,\; E_1(f_1)=E_1( e_2) = E_1 (f_2) = E_1(e_3) = E_1( f_3) =0, \\
i E_1(e_1) & = -f_3,\; E_1(f_1)=E_1( e_2) = E_1 (f_2) = E_1(e_3) = E_1( f_3) =0.
\end{align*}
We have $D_1, iD_1, E_1, i E_1 \in \AID(\Lg_\R')$ and it is easy to see that 
$\langle D_1, iD_1, E_1, i E_1 \rangle \cap \Inn(\Lg_\R')=0$, so that we obtain
\begin{align*}
\dim (\AID(\Lg_\R')) & \geq 4 + \dim (\Inn(\Lg_\R')) = 8, \\ 
\dim(\AID(\Lg_\C'))  & =\dim(\Inn(\Lg_\C'))= \dim(\Inn(\Lg_\C \oplus \Lg_\C))=4.
\end{align*}
The maps $D_1,i D_1, E_1, i E_1$ do, like any other derivation, extend to derivations of $\Lg_\C'$, but these will no longer 
be almost inner derivations.
\end{ex}

\section{Acknowledgements}
The first author was supported in part by the Austrian Science Fund (FWF), grant P28079 and grant I3248. 
The second and third author are supported by a long term structural funding, the Methusalem grant of the 
Flemish Government. The third author gratefully acknowledges financial support during his research stay at 
the Erwin Schr\"odinger International Institute for Mathematics and Physics in Vienna.

\end{document}